\documentclass[12pt]{amsart}
\usepackage{xspace,amssymb,amsfonts,epsfig,marvosym,euscript,eufrak,xypic,enumerate,amsmath,nicefrac}
\usepackage{multicol}
\usepackage[margin=3cm,footskip=25pt,headheight=20pt]{geometry}

\newcommand{\nc}{\newcommand}
\newcommand{\mc}{\mathcal}

\nc{\on}{\operatorname}
\nc{\h}{\mathfrak{h}}
\nc{\g}{\mathfrak{g}}
\nc{\n}{\mathfrak{n}}
\nc{\ch}{\on{CH}}
\nc{\wt}{\widetilde}

\nc{\F}{\mc{F}}
\nc{\C}{\mc{C}}
\nc{\M}{\on{M}}
\nc{\T}{\mc{T}}
\renewcommand{\H}{\on{H}}
\nc{\G}{\mc{G}}
\nc{\ov}{\overline}

\nc{\VFun}{\on{Vect}(\fun)}

\nc{\FF}{\mathbb{F}}
\nc{\HH}{\mathbb{H}}

\nc{\Amod}{$A$-{mod}}

\renewcommand{\H}{\mathbb{H}(\Pone)}

\nc{\Pone}{\mathbb{P}^1}
\nc{\Aone}{\mathbb{A}^1}
\nc{\fun}{\mathbb{F}_1}
\nc{\mf}{\mathfrak}
\nc{\slthat}{\widehat{\mf{sl}}_2}
\renewcommand{\a}{\mathfrak{a}}
\nc{\p}{\mathfrak{p}}
\nc{\spec}{\on{Spec}}
\nc{\Msch}{\mc{M}sch}

\theoremstyle{definition}

\newtheorem{theorem}{Theorem}
\newtheorem{lemma}{Lemma}
\newtheorem{definition}{Definition}
\newtheorem{corollary}{Corollary}
\newtheorem{remark}{Remark}

\theoremstyle{definition}
\newtheorem{example}{Example}

\begin{document}

\title{On the Hall algebra of coherent \\ sheaves on $\mathbb{P}^1$ over $\fun$.}
\author{Matt Szczesny} \thanks{The author is supported by an NSA grant}
\address{Department of Mathematics  and Statistics, 
         Boston University, Boston MA, USA}
\email{szczesny@math.bu.edu}

\begin{abstract}
We define and study the category $Coh_n(\Pone)$ of normal coherent sheaves on the monoid scheme $\Pone$ (equivalently, the $\mathfrak{M}_0$-scheme $\Pone / \fun$ in the sense of Connes-Consani-Marcolli \cite{CCM} ). This category resembles in most ways a finitary abelian category,  but is not additive. As an application, we define and study the Hall algebra of $Coh_n(\Pone)$. We show that it is isomorphic as a Hopf algebra to the enveloping algebra of the product of a non-standard Borel in the loop algebra $L {\mathfrak{gl}}_2$ and an abelian Lie algebra on infinitely many generators. This should be viewed as a $(q=1)$ version of Kapranov's result relating (a certain subalgebra of ) the Ringel-Hall algebra of $\mathbb{P}^1$ over $\mathbb{F}_q$ to a non-standard quantum Borel inside the quantum loop algebra $\mathbb{U}_{\nu} (\slthat)$, where $\nu^2=q$.
\end{abstract}

\maketitle

\section{introduction}

If $\mc{A}$ is an abelian category defined over a finite field $\FF_q$, and \emph{finitary} in the sense that $\on{Hom}(M,N)$ and $\on{Ext}^{1}(M,N)$ are finite-dimensional $\forall \; M,N \in \mc{A}$, one can attach to it an associative algebra $\HH(\mc{A})$ defined over the field $\mathbb{Q}(\nu), \; \nu=\sqrt{q}$, called the Ringel-Hall algebra of $\mc{A}$. As a $\mathbb{Q}(\nu)$--vector space, $\HH(\mc{A})$ is spanned by the isomorphism classes of objects in $\mc{A}$, and its structure constants are expressed in terms of the number of extensions between objects.  Under additional assumptions on $\mc{A}$, it can be given the structure of a Hopf algebra (see \cite{S}). 

Let $X$ be a smooth projective curve over  $\FF_q$. It is known that the abelian category $Coh(X)$ of coherent sheaves on $X$ is finitary, and one can therefore consider its Ringel-Hall algebra $\HH(X)$. This algebra was studied by Kapranov in the important paper \cite{Kap2} (see also \cite{BK}), in the context of automorphic forms over the function field $\FF_q (X)$. 
Let $L \mathfrak{sl}_2 := \mathfrak{sl}_2 \otimes \mathbb{C}[t,t^{-1}]$  be the loop algebra of $\mathfrak{sl}_2$, and $\mathbb{U}_{\nu}(L \mathfrak{sl}_2)$ the corresponding quantum loop algebra (see \cite{S}) . Denote by $ L \mathfrak{sl}^+_2$ the "positive" subalgebra  spanned by $e \otimes t^k$ and $h \otimes t^l$, $k \in \mathbb{Z}, l \in \mathbb{N}$, and let $\mathbb{U}_{\nu} ( L \mathfrak{sl}^+_2 )$ be the corresponding deformation of the enveloping algebra $\mathbb{U}(L \mathfrak{sl}^+_2)$ inside $\mathbb{U}_{\nu}(L \mathfrak{sl}_2)$.  In the case $X=\Pone$, Kapranov shows in \cite{Kap2} that there exists an embedding of bialgebras
\[
\Psi: \mathbb{U}_{\nu} (L \mathfrak{sl}^+_2) \rightarrow \HH(\Pone).
\] 

In this paper, we define and study a version of the category of coherent sheaves on the monoid scheme $\Pone$. The theory of monoid schemes is one of several models of the theory of schemes over $\fun$ - the field with one element, and was developed by Kato \cite{Kato}, Deitmar \cite{D}, Connes-Consani-Marcolli \cite{CCM}. A monoid scheme $X$ is defined as a topological space locally modeled on the spectrum of a commutative unital monoid with $0$, carrying a structure sheaf of commutative monoids $\mc{O}_X$. We define coherent sheaves on $X$ as sheaves of pointed sets carrying an action of $\mc{O}_X$ satisfying a "normality" condition, which are locally finitely generated in a suitable sense. This is a modification of the notion of coherent sheaf defined in \cite{D}. We will denote the corresponding category by $Coh_n(X_{\fun})$.  In this framework, one can make sense of most of the usual notions of algebraic geometry, such as locally free sheaves, line bundles, and torsion sheaves. The category $Coh_n(X_{\fun})$ behaves very much like a finitary abelian category except in that  $\on{Hom}(\mc{F}, \mc{G})$ only has the structure of a pointed set (which corresponds to the notion of $\fun$--vector space). 
We show that as in the case of $\Pone$ over a field, every locally free sheaf is a direct sum of line bundles, and that every coherent sheaf is a direct sum of a torsion sheaf and a torsion-free sheaf. However, over $\fun$, in addition to locally free sheaves,  there is a class of torsion-free sheaves  which we call \emph{cyclic sheaves}.   As an application, we define the Hall algebra $\HH(\Pone_{\fun})$ of $Coh_n(\Pone_{\fun})$, and describe its structure. Letting $L \mathfrak{gl}^+_2 = (\mathfrak{d} \otimes t \mathbb{C}[t]) \oplus (e \otimes \mathbb{C}[t,t^{-1}])$, where $$\mathfrak{d} := \on{span} \left\{ h_1= \left[ \begin{matrix} 1 & 0 \\ 0 & 0 \end{matrix}\right], h_2=\left[ \begin{matrix} 0 & 0 \\ 0 & 1 \end{matrix}\right] \right\} \textrm{ and } e = \left[ \begin{matrix} 0 & 1 \\ 0 & 0 \end{matrix}\right],$$ and letting $\kappa$ denote the abelian Lie algebra on generators $\{ \kappa_n \}_{n \in \mathbb{N}}$, the main result is the following:

\begin{theorem} \label{mainthm}
\[
\HH(\Pone_{\fun}) \simeq  \mathbb{U} (L \mathfrak{gl}^+_2 \oplus \kappa).
\]
\end{theorem}

This result should be naturally viewed as the $q=1$ version of Kapranov's theorem. As seen in this paper, the category $Coh_n(X_{\fun})$ is in many ways much simpler than $Coh(X_{\fun} \otimes \FF_q)$, where $X_{\fun} \otimes \FF_q$ denotes the base-change of the scheme $X_{\fun}$ to $\FF_q$. For instance, $\Pone$ over $\fun$ possesses three points - two closed points $\{ 0, \infty \}$, and a generic point $\eta$. The category $Coh(X_{\fun})$ is thus essentially combinatorial in nature. 
This suggests a possible application of the ideas of algebraic geometry over $\fun$ to the study of Hall algebras of higher-dimensional varieties.  The problem of studying the Hall algebra of $Coh(\mathbb{P}^k_{\FF_q})$ for $k > 1$ seems somewhat difficult, while that of $Coh(\mathbb{P}^k_{\fun})$ may be much more manageable. The latter should be viewed as a degenerate (i.e. $q=1$) version of the original object, and would give some hints as to its structure. 

The paper is structured as follows. In section \ref{msch} we recall basic notions related to monoid schemes and coherent sheaves on them. Section 
\ref{coh_a1} classifies normal coherent sheaves on $\Aone_{\fun}$. In section \ref{coh_p1}, we study the category $Coh_n(\Pone_{\fun})$ and establish basic structural results. In section \ref{hall_alg} we introduce the Hall algebra $\HH(\Pone_{\fun})$ of $Coh_n(\Pone_{\fun})$. Finally, in section \ref{computation} we describe the structure of $\HH(\Pone_{\fun})$ and prove Theorem \ref{mainthm}. 

\medskip

\begin{remark}
In the remainder of this paper, unless otherwise indicated, all schemes will be monoid schemes (i.e. over $\fun$), and we will omit the subscript $\fun$, which was used in the introduction for emphasis and contrast.  Thus instead of writing $X_{\fun}, \Pone_{\fun}, \Aone_{\fun}, Coh_n(X_{\fun})$, etc. we will write simply $X, \Pone, \Aone, Coh_n(X) $ etc.
\end{remark}

\medskip

\noindent{\bf Acknowledgements:} In writing this paper, the author benefited greatly from the presentation of the theory of monoid schemes found in \cite{CHWW}, and thanks C. Weibel for kindly sharing this preprint with him. He would also like to thank Oliver Lorscheid for many valuable conversations.

\section{Monoid schemes} \label{msch}

In this section, we briefly recall the notion of a monoid scheme following \cite{CHWW}. This is essentially equivalent to the notion of $\mathfrak{M}_0$-scheme in the sense of \cite{CCM}  For other approaches to schemes over $\fun$, see \cite{CC1,CC2,CCM,Du,LPL, Sou,TV}.

Recall that ordinary schemes are ringed spaces locally modeled on affine schemes, which are spectra of commutative rings.  A monoid scheme is locally modeled on an affine monoid scheme, which is the spectrum of a commutative unital  monoid with $0$. In the following, we will denote monoid multiplication by juxtaposition or $\cdot$. In greater detail:

A \emph{monoid} $A$ will be a commutative associative monoid with identity $1_A$ and zero $0_A$ (i.e. the absorbing element). We require
\[
 1_A \cdot a = a \cdot 1_A = a \hspace{1cm} 0_A \cdot a = a \cdot 0_A = 0_A \hspace{1cm} \forall a \in A
\]
Maps of monoids are required to respect the multiplication as well as the special elements $1_A, 0_A$. An \emph{ideal} of $A$ is a subset $\a \subset A$ such that $\a \cdot A \subset \a$. An ideal $\p \subset A$ is \emph{prime} if $xy \in \p$ implies either $x \in \p$ or $y \in \p$. 

Given a monoid $A$, the topological space $\spec A$ is defined to be the set $$\spec A := \{ \p | \p \subset A \textrm{ is a prime ideal } \}, $$ with the closed sets of the form $$ V(\a) := \{ \p | \a \subset \p, \p \textrm{ prime } \},  $$ together with the empty set. 
Given a multiplicatively closed subset $S \subset A$, the \emph{localization of $A$ by $S$}, denoted $S^{-1}A$, is defined to be the monoid consisting of symbols $\{ \frac{a}{s} |  a \in A, s \in S \}$, with the equivalence relation $$\frac{a}{s} = \frac{a'}{s'}  \iff \exists \; s'' \in S \textrm{ such that } as's'' = a's s'', $$ and multiplication is given by $\frac{a}{s} \times \frac{a'}{s'} = \frac{aa'}{ss'} $. 

For $f \in A$, let $S_f$ denote the multiplicatively closed subset $\{ 1, f, f^2, f^3, \cdots,  \}$. We denote by $A_f$ the localization $S^{-1}_f A$, and by $D(f)$ the open set $\spec A \backslash V(f) \simeq \spec A_f$, where $V(f) := \{ \p \in \spec A | f \in \p \}$. The open sets $D(f)$ cover $\spec A$. 
$\spec A$ is equipped with a  \emph{structure sheaf} of monoids $\mc{O}_A$, satisfying the property $\Gamma(D(f), \mc{O}_A) = A_f$.    Its stalk at $\p \in \spec A$ is $A_{\p} := S^{-1}_{\p} A$, where $S_{\p} = A \backslash \p$. 

A unital homomorphism of monoids $\phi: A \rightarrow B$ is \emph{local} if $\phi^{-1}(B^{\times}) \subset A^{\times}$, where $A^{\times}$ (resp. $B^{\times}$) denotes the invertible elements in $A$ (resp. $B$). 
A \emph{monoidal space} is a pair $(X, \mc{O}_X)$ where $X$ is a topological space and $\mc{O}_X$ is a sheaf of monoids. A \emph{morphism of monoidal spaces} is a pair $(f, f^{\#})$ where $f: X \rightarrow Y$ is a continuous map, and $f^{\#}: \mc{O}_Y \rightarrow f_* \mc{O}_X$ is a morphism of sheaves of monoids, such that the induced morphism on stalks $f^{\#}_\p : \mc{O}_{Y, f(\p)} \rightarrow f_* \mc{O}_{X, \p}$ is local. 
An \emph{affine monoidal scheme} is a monoidal space isomorphic to $(\spec A, \mc{O}_A)$. 
Thus, the category of affine monoidal schemes is opposite to the category of monoids. 
A monoidal space $(X,\mc{O}_X)$ is called a \emph{monoidal scheme}, if for every point $x \in X$ there is an open neighborhood $U_x \subset X$ containing $x$ such that $(U_x, \mc{O}_X \vert_{U_x})$ is an affine monoidal scheme. We denote by $\Msch$ the category of monoid schemes.

\begin{example} \label{P1}

Denote by $\langle t \rangle$ the free commutative unital monoid with zero generated by $t$, i.e.
\[
\langle t \rangle := \{ 0, 1, t, t^2, t^3, \cdots, t^n, \cdots \},
\]
and let $\mathbb{A}^1 := \on{Spec} \, \langle t \rangle $ - the monoidal affine line.
Let $\langle t, t^{-1} \rangle$ denote the monoid 
\[
\langle t,t^{-1} \rangle := \{ \cdots, t^{-2}, t^{-1}, 1, 0, t, t^2, t^3, \cdots \}.
\]
We obtain the following diagram of inclusions
\begin{equation*}
\langle t \rangle \hookrightarrow \langle t,t^{-1} \rangle \hookleftarrow \langle t^{-1} \rangle.
\end{equation*}
Taking spectra, and denoting by $U_0 = \on{Spec}\, \langle t \rangle , U_{\infty} = \on{Spec}\, \langle t^{-1} \rangle $, we obtain
the diagram
\begin{equation*} 
\mathbb{A}^1 \simeq U_0 \hookleftarrow U_0 \cap U_{\infty} \hookrightarrow U_{\infty} \simeq \mathbb{A}^{1}.
\end{equation*}
We define $\mathbb{P}^{1}$, the monoid projective line, to be the
monoid scheme obtained by gluing two copies of $\mathbb{A}^1$
according to this diagram. 
It has three points - two closed points $0 \in U_0, \;  \infty \in U_{\infty}$, and the generic point $\eta$. Denote by $\iota_{0} : U_0 \hookrightarrow \Pone$, $\iota_{\infty}: U_{\infty} \hookrightarrow \Pone$  the corresponding inclusions. 

\end{example}

\medskip

Given a commutative ring $R$, there exists a base-change functor
$\otimes R$ from $\Msch$ to $\on{Sch} / \on{Spec} R.$
On affine schemes, $\otimes R$ is defined by setting
$
\otimes R (\on{Spec} A) = \on{Spec} R[A], 
$ 
where $R[A]$ is the monoid algebra:
\[
R[A] := \left\{ \sum r_i a_i | a_i \in A, a_i \neq 0, r_i \in R \right\}
\]
with multiplication induced from the monoid multiplication. 
For a general monoid scheme $X$, $\otimes R$ is defined by gluing the open affine subfunctors of $X$. We denote $\otimes R (X)$ by $X \otimes R$.

\subsection{Coherent sheaves} \label{coh_sheaves}

\subsubsection{Vector spaces over $\fun$}

In this section we recall the category of vector spaces over $\fun$ following \cite{KapS, H}. 
\begin{definition}
The category $\VFun$ of vector spaces over $\fun$ is defined as follows. 
\begin{eqnarray*}
\on{Ob}(\VFun) & := & \{\textrm{ pointed sets } (V, *_V ) \} \\
\on{Hom}(V,W) & := & \{ \textrm{ maps } f: V \rightarrow W \vert \;
f(*_V) = *_W \\ & & \text{and }f \vert_{V \backslash f^{-1}(*_W)} \textrm{ is an injection }\}
\end{eqnarray*}
Composition of morphisms is defined as the composition of maps of
sets, and so is associative. We refer to the unique morphism $f \in
\on{Hom}(V,W)$ such that $f(V)= *_W$ as the zero map. 
If $V \in \VFun$ is a finite set, we define the \emph{dimension} of $V \in \VFun$ as $\dim(V) := |V| := \#V -1$ (i.e. we do not count the basepoint). 
\end{definition}

Now $\on{Hom}(V,W)$ is a pointed set, with distinguished element the zero
map. Thus \\ \mbox{$\on{Hom}(V,W) \in \VFun$.} For a fixed $V \in
\VFun$, $\on{End}(V) := \on{Hom}(V,V) $ has the structure of a
(generally) non-commutative monoid 
with $0$ (the zero map) and $1$ (the identity map).

\subsubsection{$A$--modules} \label{Amod}

Let $A$ be a monoid. An \emph{$A$--module} is a pointed set $(M,*_{M})$ together with an action 
\begin{align*}
\mu: A \times M & \rightarrow M \\
(a,m) & \rightarrow a\cdot m
\end{align*}
which is compatible with the monoid multiplication (i.e. $1_A \cdot m = m$, $a \cdot (b \cdot m) = (a \cdot b) \cdot m$), and $0_A \cdot m = *_{M} \; \forall m \in M$). 
We will refer to elements of $M \backslash *_M$ as \emph{nonzero} elements. 

A \emph{morphism of $A$--modules} $f: (M,*_M) \rightarrow (N,*_N)$ is a map of pointed sets (i.e. we require $f(*_M) = *_M$) compatible with the action of $A$, i.e. $f(a \cdot m) = a \cdot f(m)$.

A pointed subset $(M',*_M) \subset (M,*_M)$ is called an \emph{$A$--submodule} if $A \cdot M' \subset M'$. In this case we may form the quotient module $M/M'$, where $M/M' := M \backslash (M' \backslash *_M)$, $*_{M/M'} = *_M$, and the action of $A$ is defined by setting 
$$a \cdot \overline{m} = \left\{ \begin{array}{ll} \overline{a \cdot
m} & \textrm{if } a \cdot m \notin M' \\ *_{M/M'} & \textrm{ if }
a\cdot m \in M'    \end{array} \right. $$  
where $\overline{m}$ denotes $m$ viewed as an element of $M/M'$. 
If $M$ is finite, we define $|M| = \# M -1 $, i.e. the number of non-zero elements. 

\bigskip

Denote by $\Amod$ the category of $A$--modules. It has the following properties:
\begin{enumerate}
\item $\Amod$ has a zero object $\emptyset$, namely the one-element
pointed set $\{ * \}$. 
\item A morphism $f: (M,*_M) \rightarrow (N,*_N)$ has a kernel $(f^{-1}(*_N), *_M)$ and a cokernel $N/\on{Im}(f)$.
\item $\Amod$ has a symmetric monoidal structure $M \oplus N := M \vee N := M \coprod N / *_M \sim *_N$ which we will call "direct sum". 
\item If $R \subset M \oplus N$ is an $A$--submodule, then $R = (R \cap M) \oplus (R \cap N) $.  
\item $\Amod$ has a symmetric monoidal structure  $M \otimes N := M \wedge N := M \times N / \sim $, where $\sim$ is the equivalence relation generated by  $(a \cdot m, n) \sim (m, a \cdot n)$.  
\item $\oplus, \otimes$ satisfy the usual associativity and distributivity properties. 
\end{enumerate}

\bigskip

$M \in \; \Amod$ is \emph{finitely generated} if there exists a surjection $\oplus^n_{i=1} A \twoheadrightarrow M$ of $A$--modules for some $n$. Explicitly, this means that there are $m_1, \cdots, m_n \in M$ such that for every $m \in M$, $m = a \cdot m_i$ for some $1\leq i \leq n$, and we refer to the $m_i$ as \emph{generators}. $M$ is said to be \emph{free of rank n} if $M \simeq \oplus^n_{i=1} A$. 
For an element $m \in M$, define $$Ann_A (m) := \{ a \in A | a \cdot m = *_M \}.$$  Obviously $0_A \subset Ann_A (m) \;  \forall m \in M$. An element $m \in M$ is \emph{torsion} if $Ann_A (m) \neq 0_A$. The subset of all torsion elements in $M$ forms an $A$--submodule, called the \emph{torsion submodule} of $M$, and denoted $M_{tor}$. An $A$--module is \emph{torsion-free} if $M_{tor} = \{ *_M \}$ and \emph{torsion} if $M_{tor} = M$. 

\bigskip

Given a multiplicatively closed subset $S \subset A$ and an $A$--module $M$, we may form the $S^{-1} A$--module $S^{-1}M$, where
\[
S^{-1} M := \{ \frac{m}{s} \vert \; m \in M, s \in S \} 
\]
with the equivalence relation 
$$\frac{m}{s} = \frac{m'}{s'}  \iff \exists \; s'' \in S \textrm{ such that } s's'' m = s s'' m', $$
where the $S^{-1}A$--module structure is given by $\frac{a}{s} \cdot \frac{m}{s'} := \frac{am}{ss'}.$ For $f \in A$, we define $M_f$ to be $S^{-1}_f M$. 

Let $X$ be a topological space, and $\mc{A}$ a sheaf of monoids on $X$. We say that a sheaf of pointed sets $\mc{M}$ is an \emph{$\mc{A}$--module} if for every open set $U \subset X$, $\mc{M}(U)$ has the structure of an $\mc{A}(U)$--module with the usual compatibilities. In particular, given a monoid $A$ and an $A$--module $M$, there is a sheaf of $\mc{O}_{A}$--modules $\wt{M}$ on $\on{Spec} A$, defined on basic affine sets $D(f)$ by $\wt{M}(D(f)) := M_f $.  
For a monoid scheme $X$, a sheaf of $\mc{O}_X$--modules $\F$ is said to be \emph{quasicoherent} if for every $x \in X$  there exists an open affine $U_x \subset X$ containing $x$ and an $O_X (U_x)$--module $M$ such that $\F \vert_{U_x} \simeq \wt{M}$.  $\F$ is said to be \emph{coherent} if $M$ can always be taken to be finitely generated, and \emph{locally free} if $M$ can be taken to be free. Please note that here too our conventions are different from \cite{D}. For a monoid $A$, there is an equivalence of categories between the category of quasicoherent sheaves on $\on{Spec} A$ and the category of $A$--modules, given by $\Gamma(\on{Spec} A, \cdot)$. A coherent sheaf $\F$ on $X$ is \emph{torsion} (resp. \emph{torsion-free}) if $\F(U)$ is a torsion $\mc{O}_X (U)$--module (resp. torsion-free $\mc{O}_X (U)$--module ) for every open affine $U \subset X$. 
 If $X$ is connected, we can define the \emph{rank} of a locally free sheaf $\F$ to be the rank of the stalk $\F_{x}$ as an $\mc{O}_{X,x}$--module for any $x \in X$. A locally free sheaf of rank one will be called a \emph{line bundle}. 

\medskip

\begin{remark} \label{ttf}
It is clear that if $\F$ is torsion and $\F'$ torsion-free, then $\on{Hom}(\F, \F') = 0$. 
\end{remark}

\bigskip

\begin{remark} \label{subobjects}
It follows from property $(4)$ of the category $\Amod$ that if $\F,\F'$ are quasicoherent $\mc{O}_X$--modules, and $\mc{G} \subset \F\oplus \F'$ is a submodule, then \mbox{$\mc{G} = (\mc{G}\cap \F) \oplus (\mc{G} \cap \F' )$}, where for an open subset $U \subset X$,  
\begin{equation*} (\mc{G} \cap \F) (U) := \mc{G}(U) \cap \F (U). \end{equation*}
\end{remark}

\bigskip

If $X$ is a monoid scheme, we will denote by $Coh(X)$ the category of coherent \mbox{$\mc{O}_X$--modules} on $X$. It follows from the properties of the category $\Amod$ listed in section \ref{coh_sheaves} that $Coh(X)$ possesses a zero object $\emptyset$ (defined as the zero module $\emptyset$ on each open affine $\spec A \subset X$), kernels and cokernels, as well as monoidal structures $\oplus$ and $\otimes$. We may therefore talk about exact sequence in $Coh(X)$. A short exact sequence isomorphic to one of the form
\[
\emptyset \rightarrow \F \rightarrow \F \oplus \mc{G} \rightarrow \mc{G} \rightarrow \emptyset
\]
is called \emph{split}.  A coherent sheaf $\F$ which cannot be written as $\F=\F' \oplus \F''$ for non-zero coherent sheaves is called \emph{indecomposable}. A coherent sheaf containing no non-zero proper sub-sheaves is called \emph{simple}.

\medskip

If 
\[
\emptyset \rightarrow \F \overset{j_1}{\rightarrow} \F'  \overset{p_1}{\rightarrow} \F'' \rightarrow \emptyset
\]
and
\[
\emptyset \rightarrow \G \overset{j_2}{\rightarrow} \G'  \overset{p_2}{\rightarrow} \G'' \rightarrow  \emptyset
\]
are two short exact sequences in $Coh(X)$, we say that the short exact sequence 
\[
\emptyset \rightarrow \F \oplus \G \overset{j_1 \oplus  j_2}{\rightarrow} \F' \oplus \G' \overset{p_1 \oplus p_2}{\rightarrow} \F'' \oplus \G'' \rightarrow \emptyset
\]
is their \emph{direct sum}. 

\medskip

\begin{remark} \label{subobj2}
Because of the property of $Coh(X)$ discussed in Remark \ref{subobjects}, it follows that if $\F$ is an indecomposable coherent sheaf, and $\F \subset \F' \oplus \F''$, then $\F \subset \F'$ or $\F \subset \F''$. 
\end{remark}

Every coherent sheaf can be written as a direct sum of indecomposable coherent sheaves, and this decomposition is unique up to permutation of the factors. 

\medskip

We will make use of the gluing construction for coherent sheaves. Namely, suppose that $\mathfrak{U} = \{ U_i \}_{i \in I}$ is an affine open cover of $X$, and suppose we are given for each $i \in I$ a sheaf $\F_i$ on $U_i$, and for each $i,j \in I$ an isomorphism $\phi_{ij}: \F_i \vert_{U_i \cap U_j} \rightarrow \F_{j} \vert_{U_i \cap U_j}$ such that 
\begin{enumerate}
\item $\phi_{ii} = id$
\item For each $i,j,k \in I$, $\phi_{ik} = \phi_{jk} \circ \phi_{ij} $ on $U_i \cap U_j \cap U_k$
\end{enumerate}
Then there exists a unique sheaf $\F$ on $X$ together with isomorphisms $\psi_i: \F \vert_{U_i} \rightarrow \F_i$, such that for each $i,j \; \psi_j = \phi_{ij} \circ \psi_i$ on $U_i \cap U_j$. If moreover the $\F_i$ are coherent and $\phi_{ij}$ isomorphisms of coherent $\mc{O}_X$--modules, then $\F$ is itself coherent. 

\subsection{Normal modules, coherent sheaves, and exact sequences}

\begin{definition}
An $A$--module $(M,*_M)$ is \emph{normal} if it is an $A$--module in the category $\VFun$. More precisely, we require that for each $a \in A$, the map
\begin{align*}
l_a : M & \rightarrow M \\
m & \rightarrow a \cdot m
\end{align*}
be a morphism in $\VFun$  - i.e. in particular 
$l_a\vert_{M \backslash l_a^{-1} (*_M) }$ 
is an injection. A morphism $$f: (M,*_M) \rightarrow (N,*_N)$$ of $A$--modules is \emph{normal} if it is a morphism in $\VFun$ - i.e. 
$f\vert_{M \backslash f^{-1} (*_N) }$ is an injection. 
\end{definition}

\begin{definition}
Let $X$ be a monoid scheme, and $\F$ a coherent sheaf on $X$. $\F$ is \emph{normal} if for every open affine $U \subset X$, $\F = \wt{M}$ for a normal $\mc{O}_X (U)$--module $(M,*_M)$. A morphism of coherent sheaves $\phi: \F \rightarrow \F'$ is \emph{normal} if its restriction to every open affine $U \subset X$ is given by a normal morphism of $\mc{O}_X (U)$--modules. A short exact sequence of coherent sheaves
\[
\emptyset \rightarrow \F \rightarrow \F' \rightarrow \F'' \rightarrow \emptyset
\]
is \emph{normal} if $\F, \F', \F''$ are normal coherent sheaves, and all morphisms are normal. 
\end{definition}

Normal coherent sheaves and morphisms form a (non-full) subcategory of $Coh(X)$ possessing a zero object, kernels, cokernels, and closed under the operations $\oplus, \otimes$. We denote this subcategory by $Coh_n (X)$. 

\begin{remark}
The reason to restrict attention to the category $Coh_n (X)$ is that it is better behaved with respect to extensions. For example, even on $\Pone$, given coherent sheaves $\F, \F'$, $\on{Ext}^1 (\F,\F')$ may well be infinite (as a set), which introduces problems in the definition of the Hall algebra below. 
\end{remark}

From now on, we will focus primarily on the category $Coh_n (X)$. 

\section{Normal coherent sheaves on $\Aone$} \label{coh_a1}

In this section, we proceed to classify normal coherent sheaves on $\Aone = \spec \langle t \rangle $. By the results of section \ref{Amod}, each such corresponds to a finitely generated $\langle t \rangle$--module $(M, *_M)$. To $M$, we may attach a directed graph $\wt{\Gamma}_M$, whose vertices correspond to elements of $M$, and where we draw an arrow from $m$ to $t \cdot m$. The graph $\wt{\Gamma}_M$ has the property that
each vertex except for $*_M$ has at most one incoming arrow and exactly one outgoing arrow (this follows from normality). Let $\Gamma_M$ be the directed graph obtained from $\wt{\Gamma}_M$ by removing the vertex $*_M$ and all arrows leading to it. The connected components of $\Gamma_M$ correspond to normal indecomposable $\langle t \rangle$--modules. Let us consider the following three types of $\langle t \rangle$--modules:

\begin{enumerate}
\item The free module $\langle t \rangle$ viewed as a module over itself. The graph $\Gamma_{\langle t \rangle}$ consists of an infinite ladder with initial vertex $1$. $\langle t \rangle$ is torsion-free.
\item For $n \in \mathbb{N}$, let $E_n$ denote the $\langle t \rangle$--submodule $\{0, t^n, t^{n+1}, \cdots \}$ of $\langle t \rangle$. Let $T_n$ denote the quotient $\langle t \rangle$--module $\langle t \rangle/ E_n$. The graph $\Gamma_{T_n}$ is an $n$-step ladder. $T_n$ is a torsion module.

\item For $n \in \mathbb{N}$, let $C_n$ denote the $\langle t \rangle$--module whose elements are $\{ z_{[i]}, * \},  \; [i] \in \mathbb{Z}/n \mathbb{Z}$ (where the basepoint is denoted by $*)$, and where the action of $\langle t \rangle$ is given by $t \cdot z_{[i]} = z_{[i+1]}$, and $0$ sends everything to $*$. The graph $\Gamma_{C_n}$ is a directed cycle with $n$ vertices. $C_n$ is torsion-free. 
\end{enumerate}

\begin{remark} \label{tor_submodule}
If $m \leq n$, then the $\langle t \rangle$--module $T_n$ contains a unique submodule isomorphic to $T_m$ ( given by $\{ t^{n-m}, t^{n-m+1}, \cdots, t^{n-1}, 0 \}$ ) and we have a short exact sequence 
\begin{equation*}
\emptyset \rightarrow T_m \rightarrow T_n \rightarrow T_{n-m} \rightarrow \emptyset
\end{equation*}
\end{remark}

\begin{lemma}
\begin{enumerate}
\item Every indecomposable normal $\langle t \rangle$--module is one of $\langle t \rangle$, $T_n$, or $C_n$. 
\item Every finitely generated normal $\langle t \rangle$--module $M$ is of the form $$  \langle t \rangle^{\oplus m} \oplus^k_{i=1} T_{n_i} \oplus^l_{j=1} C_{n_j} $$
\end{enumerate}
\end{lemma}

\begin{proof}
For (a), it suffices to classify connected directed graphs $\Gamma$
having the properties: every vertex of $\Gamma$ has at most one incoming 
arrow, and every vertex of $\Gamma$ has at most one outgoing arrow. It 
is immediate that the three possibilities enumerated above are the only ones.
Part (b) is now immediate from the decomposition of the graph
$\Gamma_M$ into connected components. 
%
%
\end{proof}


\section{Normal coherent sheaves on $\Pone$} \label{coh_p1}

In this section, we study normal coherent sheaves on $\Pone$. Recall from example \ref{P1} that $\Pone$ is obtained by gluing two copies of $\mathbb{A}^1$ labeled $U_{0}, U_{\infty}$, and has three points - $0, \infty, \textrm{ and } \eta$, where the first two are closed, and $\eta$ is the generic point. 
The data of a coherent sheaf on $\Pone$ therefore corresponds to a pair of coherent sheaves $\F_0, \F_{\infty}$ on $U_0, U_{\infty}$, together with a clutching isomorphism $$ \tau: \F_0 \vert_{U_0 \cap U_\infty} \rightarrow  \F_{\infty} \vert_{U_0 \cap U_\infty}. $$ In other words, a $\langle t \rangle$-- module $M_0$, a  $\langle t^{-1} \rangle$-- module $M_{\infty}$, and a $\langle t, t^{-1} \rangle$--equivariant isomorphism 
$$ \wt{\tau}: S_0^{-1} M_0 \rightarrow S_{\infty}^{-1} M_{\infty},$$ where $S_0 = \{1, t, t^2, \cdots \} \subset \langle t \rangle$ and $S_{\infty} = \{1, t^{-1}, t^{-2}, \cdots \} \subset \langle t^{-1} \rangle$. The resulting coherent sheaf $\F$ if normal if and only if $M_0, M_{\infty}$ are normal. 

Since the only open set of $\mathbb{A}^1$ containing the closed point is all of $\Aone$, any locally free sheaf on $\Aone$ is trivial. Observe furthermore that $\on{Aut}_{\mc{O}_{\Aone}} (\mc{O}^{\oplus n}_{\Aone}) = S_n$ (i.e. any automorphism of a free module is determined by a permutation of the generators). 

\subsection{Building blocks of $Coh_n(\Pone)$}

We proceed to introduce the building blocks of $Coh_n(\Pone)$ - the line bundles $\mc{O}(n)$, the cyclic sheaves $\mc{C}_n$,  and the torsion sheaves $\T_{x,n}, \; x = 0, \infty$. 

\begin{itemize}
\item Since $\on{Aut}(\mc{O}_{\Aone}) = \{ 1 \}$, the data of a line bundle on $\Pone_{\fun}$ corresponds to an $\mc{O}_{\Pone}$--linear clutching isomorphism
\[
\tau : \mc{O}_{U_0} \vert_{U_0 \cap U_\infty} \rightarrow \mc{O}_{U_\infty} \vert_{U_0 \cap U_\infty}
\]
i.e. an automorphism $\tau$ of $\{ \cdots t^{-2}, t^{-1}, 0, 1, t, t^2, \cdots \}$ equivariant with respect to the action of $\langle t,t^{-1} \rangle$ - these are of the form $\tau_m(t^k) = t^{k-m}$, and are therefore determined by an integer $m \in \mathbb{Z}$. Denote by $\mc{O}(m)$ the line bundle obtained from the clutching map $\tau_m$ (we have chosen the convention so that $\Gamma(\Pone, \mc{O}(m)) \neq \emptyset$  for $m \geq 0$).  

\item We may glue two copies of $C_n$, one on $U_0$ and one on $U_{\infty}$ along $U_0 \cap U_{\infty}$. Let $C^{0}_n$ be the $\langle t \rangle$--module $C_n$ from section \ref{coh_a1}, and $C^{\infty}_n$ the $\langle t^{-1} \rangle$--module with elements $\{ \wt{z}_{[j]}, * \}, \; j \in \mathbb{Z} / n \mathbb{Z}$, where $t^{-1} \cdot \wt{z}_{[j]} = \wt{z}_{[j-1]}$, and $0$ sends everything to $*$. Note that any element of $S_0^{-1} C^0_n$ (resp. $S_{\infty}^{-1} C^{\infty}_n$) is equivalent to one of the form $z_{[i]} = \frac{z_{[i]}}{1}$ (resp. $\wt{z}_{[j]} = \frac{\wt{z}_{[j]}}{1}$ ).  An isomorphism of $\langle t, t^{-1} \rangle$--modules
\[
\wt{\tau}: S_0^{-1} C^0_n \rightarrow S_{\infty}^{-1} C^{\infty}_n
\]
is determined by where it sends the element $z_{[0]} = \frac{z_{[0]}}{1}$, and the image $\wt{\tau} (z_{[0]})$ may be any non-zero element $\wt{z}_{[j]}$ of $S_{\infty}^{-1} C^{\infty}_n$. Let $\mc{C}_n [m]$ be the normal coherent sheaf on $\Pone$ given by the clutching isomorphism $\wt{\tau}_m(z_{[0]}) = \wt{z}_{[m]}$, $m \in \mathbb{Z}/ n \mathbb{Z}$. I claim that $\mc{C}_n [m] \simeq \mc{C}_n [m']$ for any $m, m' \in \mathbb{Z}/ n \mathbb{Z}$. The data of an isomorphism $$ \phi: \mc{C}_n [m] \rightarrow \mc{C}_n [m'] $$ is equivalent to the data of isomorphisms 
\begin{align*}
\phi_0 : C^0_n & \rightarrow C^0_n \\
\phi_{\infty} : C^{\infty}_n & \rightarrow C^{\infty}_n
\end{align*}
satisfying the property
\begin{equation*} \wt{\tau}_{m'} \circ S_0^{-1} \phi_0 =
S^{-1}_{\infty} \phi_{\infty} \circ \wt{\tau}_m. \label{compat} 
\end{equation*}
Let $k_0, k_{\infty} \in \mathbb{Z}/ n \mathbb{Z}$ be such that $k_0 + m' \equiv k_{\infty} + m \mod n.$ Then we may define $\phi_0 (z_{[i]}) = z_{[i+k_0]}$, and $\phi_{\infty}(\wt{z}_{[j]}) = \wt{z}_{[j+k_{\infty}]}$. One checks easily that the condition \ref{compat} is satisfied, so that $\phi_0, \phi_{\infty}$ indeed glue to yield an isomorphism. We denote the corresponding isomorphism class of normal coherent sheaf on $\Pone$ by $\mc{C}_n$. $\mc{C}_n$ is torsion-free and simple. 

\item Since $T_n$ is a torsion sheaf, we have $S^{-1}_0 T_n = 0$, which implies that a sheaf on $U_0$ isomorphic to $T_n$ can only be glued to the $0$ sheaf on $U_{\infty}$ (and similarly with the roles of $U_0, U_{\infty}$ interchanged ). We thus obtain torsion sheaves $\T_{0,n}, \T_{\infty, n}$ on $\Pone$ supported at $0, \infty$ respectively.


\end{itemize}

\medskip

\begin{remark} It is clear that each of $\mc{O}(n), \mc{C}_{m}, \mc{T}_{0,k}, \mc{T}_{\infty, l}$ is indecomposable. \end{remark}

\bigskip

We are now ready to classify normal coherent sheaves on $\Pone$. 

\begin{lemma} \label{gluing_lemma}
Let $X$ be a monoid scheme, $\{ \F_i \}_{i \in I}$, $\{ \mc{G}_{j} \}_{j \in J}$ finite collections of indecomposable normal coherent sheaves on $X$,
 and $$\psi: \oplus_{i \in I} \F_i \rightarrow \oplus_{j \in J} \mc{G}_j $$ an isomorphism. Then $\vert I \vert = \vert J \vert$, and $\psi$ is given by collection of isomorphisms $\psi_i : \F_i \rightarrow \mc{G}_{j_i}, \; i \in I$. 
\end{lemma}

\begin{proof}
We may reduce to the case where $X$ is affine, and proceed by induction on $\vert I \vert$. The statement is obvious when $\vert I \vert = 1$. Suppose the statement holds for $\vert I \vert \leq k$, and suppose that $\vert I \vert = k+1$. Let $i_1 \in I$, then by Remark \ref{subobj2}, we have
\[
\psi(\F_{i_1}) = \oplus_{j \in J} \psi(\F_{i_1}) \cap \mc{G}_j
\]
Since $\psi(\F_{i_1})$ is indecomposable, we have $\psi(\F_{i_1}) \subset \mc{G}_{j_{i_1}}$ for some $j_{i_1} \in J$. We thus have $\psi^{-1} (\mc{G}_{j_{i_1}}) = \F_{i_1} \oplus \mc{K}$ for some normal coherent sheaf $\mc{K}$. Since $\mc{G}_{j_{i_1}}$ is indecomposable, $\mc{K} = 0$, and $\psi \vert_{\F_{i_1}}: \F_{i_1} \rightarrow \mc{G}_{j_{i_1}}$ is an isomorphism. Thus, $\psi \vert \oplus_{i \neq i_1} \F_i \rightarrow \oplus_{j \neq j_{i_1}} \mc{G}_j$ is an isomorphism, and the statement follows by induction. 
\end{proof}

\begin{theorem} \label{indecomposables}
Every normal coherent sheaf on $\Pone$ is a direct sum of the indecomposable
normal coherent sheaves $\mc{O}(n), \mc{C}_{m}, \mc{T}_{0,k},
 \mc{T}_{\infty, l}$, for $n \in \mathbb{Z}$, \; $m, k, l \in \mathbb{N}$. 
\end{theorem}

\begin{proof}
This follows from Lemma \ref{gluing_lemma}, with $X = U_0 \cap U_{\infty}$, $\psi = {\tau}$, and the collections of normal coherent sheaves $\{ \F_i \}, \{ \mc{G}_j \}$ corresponding to the $\langle t, t^{-1} \rangle$ modules $S_0^{-1} M_0$ and $S_{\infty}^{-1} M_{\infty}$ respectively. 
\end{proof}

\begin{corollary}
\begin{enumerate}
\item Every normal coherent sheaf on $\Pone$ is of the form
\begin{equation*} 
\F = \F_{tor} \oplus \F_{tf}
\end{equation*}
with $\F_{tor}$ a torsion coherent sheaf, and $\F_{tf}$ torsion-free. 
\item $\F_{tor}$ is a direct sum of sheaves $\mc{T}_{0,k},
\mc{T}_{\infty,l}$ for $k,l \in \mathbb{N}$, and $\F_{tf}$ is a direct sum 
of sheaves $\mc{O}(n), \mc{C}_m$ for $n \in \mathbb{Z}, m \in\mathbb{N}$. 
\item Every locally free coherent sheaf on $\Pone$ is a direct sum of $\mc{O}(n)$. 
\end{enumerate}
\end{corollary}

\medskip

\begin{lemma} \label{ninjection}
Let $X$ be a monoid scheme, and  $\F, \{ \mc{G}_j \}_{j \in J}$ indecomposable normal coherent sheaves on $X$. Then any exact sequence of the form 
\[
\emptyset \rightarrow \F \overset{f}{\rightarrow} \oplus_{j \in J} \mc{G}_j  \overset{g}{\rightarrow} \mc{H} \rightarrow \emptyset
\]
is a direct sum of exact sequences
\begin{equation*} 
\emptyset \rightarrow \F \overset{f}{\rightarrow} \mc{G}_r  \overset{p}{\rightarrow} \mc{G}_r / \F \rightarrow \emptyset
\end{equation*}
and
\begin{equation*}
\emptyset \rightarrow \emptyset {\rightarrow} \oplus_{j \neq r, j \in J} \mc{G}_j  \overset{q}{\rightarrow}  \oplus_{j \neq r, j \in J} \mc{G}_j  \rightarrow \emptyset
\end{equation*}
where $r \in J$, $f$ is an injective map of coherent sheaves, and $q$ an isomorphism. 
\end{lemma}

\begin{proof}
We may reduce to the case where $X$ is affine. By Remark \ref{subobj2}, $f$ maps $\F$ injectively to a subsheaf of some $\mc{G}_r, r \in J$. The rest of the Lemma follows from this observation.
\end{proof}

\medskip

\begin{lemma} \label{homs}
\begin{align}
\label{part1} \on{Hom}(\T_{x,n}, \F)  &= 0  \textrm{ if } \F \textrm{ is torsion-free } \\
\label{part3} \on{Hom}(\T_{x,n}, \T_{x',m}) &= 0 \textrm{ if } x \neq x' \\
\label{part4} \on{Hom}(\mc{C}_n, \mc{C}_m) &= 0 \textrm{ if } n \neq m \\ 
\label{part5} \on{Hom}(\mc{C}_n, \mc{T}_{x,m}) &= 0 \\ 
\label{part6}  \on{Hom}(\mc{C}_n, \mc{O} (m) ) &= 0 \\ 
\label{part7}  \on{Hom}(\mc{O}(m), \mc{C}_n ) &= 0 
\end{align}
\end{lemma}

\begin{proof} 
Part 
(\ref{part1}) follows from remark \ref{ttf}. (\ref{part3}) follows from the observation that $\T_{x,n}$ and $\T_{x',m}$ have disjoint supports. (\ref{part4})  follows from the fact that $\mc{C}_n$ are simple. 
Finally, parts (\ref{part5}), (\ref{part6}), (\ref{part7}) follow by
restricting to $U_0, U_{\infty}$ and the corresponding statements for
$\langle t \rangle, \langle t^{-1} \rangle$--modules. 
\end{proof}

\begin{lemma} \label{oinj}
\begin{align*}
\vert \on{Hom}( \mc{O}(n), \mc{O}(m) ) \vert & = \left\{ \begin{array}{l} 0 \textrm{ if $ n > m $ } \\ m - n +1 \textrm{ if $n \leq m$ } \end{array} \right. 
\end{align*}
\end{lemma}

\begin{proof}
Let $\rho \in \on{Hom}(\mc{O}(n), \mc{O}(m))$, and let $\rho_0$ (resp. $\rho_{\infty}$) denote the restriction of $\rho$ to $U_0$ (resp. $U_{\infty}$). Trivializing $\mc{O}(n), \mc{O}(m)$ on $U_0, U_{\infty}$, we may write $\rho_0 (t^k) = t^{k+a_0}$, $\rho(t^l) = t^{l + a_{\infty}}$, where $a_0 \geq 0$, $a_{\infty} \leq 0$. The condition $\rho_0 \vert_{U_0 \cap U_{\infty}} = \rho_{\infty} \vert_{U_0 \cap U_{\infty}}$ becomes 
$$ \sigma_m \circ \rho_0 = \rho_{\infty} \circ \sigma_n,$$
which reduces to $a_0 - a_{\infty} = m -n .$ It is clear that the number of solutions to this equation with $a_0 \geq 0, a_{\infty} \leq 0$ is $m-n+1$ if $m \geq n$ and $0$ otherwise. 
\end{proof}

\begin{remark} \label{rmk2}
It follows from the proof of the last Lemma that if
\[
\emptyset \rightarrow \mc{O}(n) \rightarrow \mc{O}(m) \rightarrow \T \rightarrow \emptyset
\]
is a short exact sequence, then $\T \simeq \T_{0,k_0} \oplus
\T_{\infty,k_{\infty}}$ with $k_0 + k_{\infty} = m - n \ge0$. 
There is precisely one such short exact sequence with fixed cokernel. 
\end{remark}

\bigskip

\begin{lemma} \label{tor_maps}
If $x \in \{ 0, \infty \}$, then 
$$\vert \on{Hom}(\T_{x,m}, \T_{x,n}) \vert =  \left\{ \begin{array}{l} 1 \textrm{ if } n \geq m   \\ 0 \textrm{ if } n < m  \end{array} \right. $$
\end{lemma}

\begin{proof}
This follows from Remark \ref{tor_submodule}. 
\end{proof}

\begin{theorem} \label{ses_classification}
Every short exact sequence in $Coh_n(\Pone)$ is a direct sum of short exact sequences of the form
\begin{enumerate}
\item $$ \emptyset \rightarrow \T_{x,m} \rightarrow \T_{x,n} \rightarrow \T_{x,n-m} \rightarrow \emptyset, \; x \in \{0, \infty \}, $$
\item $$ \emptyset \rightarrow \mc{O}(n) \rightarrow \mc{O}(m) \rightarrow \T \rightarrow \emptyset $$ as in Remark \ref{rmk2}
\item Split short exact sequences. 
\end{enumerate}
\end{theorem}

\begin{proof}
Let 
\begin{equation*} 
\emptyset \rightarrow \F \rightarrow \F' \rightarrow \F'' \rightarrow \emptyset
\end{equation*}
be a short exact sequence in $Coh_n(\Pone)$. 
By Theorem \ref{indecomposables}, 
$\F$ can be uniquely written as a sum of indecomposable factors $\mc{O}(n), \mc{C}_m, \T_{0,k}, \T_{\infty,l}$, and we proceed by induction on the number $r$ of these. If $r=1$, then the statement follows from Remark \ref{subobj2}, and Lemmas \ref{homs}, \ref{oinj}, and \ref{tor_maps}. Suppose now the theorem holds for $r \leq p$, $p \geq 1$, and $\F$ has $p+1$ indecomposable factors. Write $\F = \mc{I} \oplus \mc{H}$ with $\mc{I}$ indecomposable. By Remark \ref{subobj2}, $\mc{I}$ maps into an indecomposable factor $\mc{J}$ of $\F'$, and is moreover the \emph{only} factor of $\F$ to map into $\mc{J}$. Writing $\F' = \mc{J} \oplus \mc{K}$, it follows that 
the sequence 
is a direct sum of \[
\emptyset \rightarrow \mc{I} \rightarrow \mc{J} \rightarrow \mc{J} / \mc{I} \rightarrow \emptyset
\]
and
\[
\emptyset \rightarrow \mc{H} \rightarrow \mc{K} \rightarrow \mc{K} / \mc{H} \rightarrow \emptyset
\]
and the result follows from the inductive hypothesis.
\end{proof}

\bigskip

\begin{corollary} \label{finitary}
Let $\F,\mc{G}$ be normal coherent sheaves on $\Pone$. Then
\begin{enumerate}
\item $\vert \on{Hom}(\F,\mc{G}) \vert < \infty$ \label{hom}
\item The number of extensions of $\F$ by $\mc{G}$ is finite. \label{ext}
\end{enumerate}
\end{corollary}

\begin{proof}
Let $$\F = \oplus^{k}_{i=1} \F_{i} \textrm{ and } \mc{G}= \oplus^l_{j=1} \mc{G}_j,  $$ where $\F_i, \mc{G}_j$ are indecomposable. Since any $\phi \in \on{Hom}(\F,\mc{G})$ is determined by its restriction $\phi_i$ to each $\F_i$, and each $\phi_i$ by its projections  $\phi_{ij} \in \on{Hom}(\F_i,\mc{G}_j)$, we have
\[
 \on{Hom}(\F,\mc{G}) \subset \oplus_{i,j}  \on{Hom}(\F_i,\mc{G}_j) 
\]
By Lemmas \ref{homs}, \ref{oinj}, and \ref{tor_maps}, each $\on{Hom}(\F_i,\mc{G}_j)$ is finite, and this proves part (\ref{hom}). Part (\ref{ext}) follows from Theorem \ref{ses_classification}. 
\end{proof}

\bigskip

Denote by $\on{Iso}(Coh_n(\Pone))$ the set of isomorphism classes of normal coherent sheaves on $\Pone$.
Define the Grothendieck group $K_0(\Pone)$ of $Coh_n(\Pone)$ by
\[
K_0(\Pone) := \mathbb{Z}[\on{Iso}(Coh_n(\Pone))] / \sim
\]
where $\sim$ is the subgroup  generated by $ [\F] + [\F''] - [\F']$
for 
each short exact sequence $0\to\F\to\F'\to\F''\to0$.
Define a homomorphism of abelian groups
\[
\Psi: \mathbb{Z}[\on{Iso}(Coh_n(\Pone))] \rightarrow \mathbb{Z} \oplus \mathbb{Z} \oplus (\oplus^{\infty}_{i=1} \mathbb{Z})
\]
on generators by
\begin{align*}
\Psi([\mc{O}(k)]) &= (1,k,\underline{0}) \\
\Psi([\T_{x,n}]) &= (0,n,\underline{0}) \\
\Psi([\mc{C}_m]) &= (0,0,e_m)
\end{align*}
where $\underline{0}$ denotes the $0$ vector in $\oplus^{\infty}_{i=1} \mathbb{Z}$, and $e_m$ the one with a $1$ in the $m$-th position, and $0$ everywhere else.
Since $\Psi$ is additive on every short exact sequence in Theorem \ref{ses_classification}, it follows that $\Psi$ descends to $K_0(\Pone)$, and it is easy to see that it is an isomorphism. We have thus shown:

\begin{theorem} \label{Knot}
$K_0(\Pone) \simeq \mathbb{Z} \oplus \mathbb{Z} \oplus (\oplus^{\infty}_{i=1} \mathbb{Z})$.
\end{theorem}

We call the first factor \emph{rank} and the second \emph{degree} in analogy with the case of $\Pone$ over a field.

\section{Hall algebras} \label{hall_alg}

In this section, we introduce the Hall algebra $\H$ of the category $Coh_n(\Pone)$. For more on Hall algebras see \cite{S}. 
As a vector space:
\begin{equation*} 
\HH(\Pone) := \{ f: \on{Iso}(Coh_n(\Pone)) \rightarrow \mathbb{C} \; \vert \; \# \on{supp} (f) < \infty \}.
\end{equation*}
We equip $\H$ with the convolution product
\begin{equation*} 
f \star g (\F) = \sum_{\F' \subset \F} f(\F / \F') g(\F'),
\end{equation*}
where the sum is over all coherent sub-sheaves $\F'$ of the isomorphism class $\F$ (in what follows, it is conceptually helpful to fix a representative of each isomorphism class). Note that Corollary \ref{finitary} and the finiteness of the support of $f,g$ ensures that the sum in 
the convolution product 
is finite.

\begin{lemma}
The convolution product $*$ is associative. 
\end{lemma}

\begin{proof}
Suppose $f,g,h \in \H$. Then
\begin{align*}
(f \star (g \star h)) (\F) &= \sum_{\F' \subset \F} f(\F/\F') (g \star h)(\F') \\
                                   &= \sum_{\F' \subset \F} f(\F/ \F') (\sum_{\F'' \subset \F'} g(\F'/\F'') h(\F''))  \\
                                   &= \sum_{\F'' \subset \F' \subset \F} f(\F/ \F') g(\F'/\F'') h(\F'')
\end{align*}
whereas
\begin{align*}
((f \star g) \star h ) (\F) &= \sum_{\F'' \subset \F} (f \star g) (\F / \F'') h(\F'') \\
                                    &= \sum_{\F'' \subset \F} (\sum_{\mc{K} \subset \F / \F''} f((\F / \F'') / \mc{K}) g(K) ) h(\F'') \\
                                    &= \sum_{\F'' \subset \F' \subset \F} f(\F/\F') g(\F'/\F'') h(\F''),
\end{align*}
where in the last step we have used the fact that there is an inclusion-preserving bijection between sub-sheaves $\mc{K} \subset \F / \F''$ and sub-sheaves $\F' \subset \F$ containing $\F''$, under which $\F' / \F'' \simeq \mc{K}$.  This bijection is compatible with taking quotients in the sense that $(\F/\F'')/\mc{K} \simeq \F/\F'$. 
\end{proof}

We may also equip $\H$ with a coproduct
\begin{equation*}
\Delta: \H \rightarrow \H \otimes \H
\end{equation*}
given by
$ 
\Delta(f)(\F,\F') := f(\F \oplus \F').
$ 
The coproduct $\Delta$ is clearly co-commutative.

\begin{lemma}
The following holds in $\H$:
\begin{enumerate}
\item $\Delta$ is co-associative: $(\Delta \otimes \on{Id}) \circ \Delta = (\on{Id} \otimes \Delta) \circ \Delta$
\item $\Delta$ is compatible with $\star$:  $\Delta(f \star g) = \Delta(f) \star \Delta(g).$
\end{enumerate}
\end{lemma}

\begin{proof}
The proof of both parts is the same as the proof of the corresponding statements for the Hall algebra of the category of quiver representations over $\fun$, given in \cite{Sz}. 
\end{proof}

We may equip $\H$ with a grading by $K^+_0 (\Pone )$ - the effective cone inside $K_0 (\Pone)$, (which by Theorem \ref{Knot} is isomorphic to $\mathbb{N} \times \mathbb{N} \times \mathbb{N} $) as follows.
$\H$ is spanned by $\delta$-functions $\delta_{\F}$ supported on individual isomorphism classes, and we define
\[
\deg(\delta_{\F}) = [\F] \in K_0 (\Pone)
\]
where $[\F]$ denotes the class of $\F$ in the Grothendieck group. With this grading, $\H$ becomes a graded, connected, co-commutative Hopf algebra. By the Milnor-Moore Theorem, $\H$ is isomorphic to  $\mathbb{U}(\mathfrak{q})$ - the universal enveloping algebra of $\mathfrak{q}$, where the latter is the Lie algebra of its primitive elements. The definition of the co-product 
implies that $\mathfrak{q}$ is spanned  by $\delta_{\F}$ for isomorphism classes $\F$ of indecomposable coherent sheaves, which by Theorem \ref{indecomposables} are $\mc{O}(n), \mc{C}_{m}, \mc{T}_{0,k}, \mc{T}_{\infty, l}$, for $n \in \mathbb{Z}$, \; $m, k, l \in \mathbb{N}$. 

\section{The structure of $\H$ } \label{computation}

In this section we compute the structure of the Hopf algebra $\H$. We will use the shorthand notation $\ov{\F}$ for $\delta_{\F}$ - the delta-function supported on the isomorphism class of $\F$. The following theorem follows from Theorem \ref{indecomposables} and Theorem \ref{ses_classification}:

\begin{theorem}\label{HP1_identities}
The following identities hold in $\H$:
\begin{align*}
\ov{\mc{O}(n)} \cdot \ov{\mc{O}(m)} & = \ov{\mc{O}(n) \oplus \mc{O}(m) } \; \textrm{ if } m \neq n \\
\ov{\T_{x,n}} \cdot \ov{\mc{O}(m)} &=  \ov{\mc{O}(n+m)} + \ov{\T_{x,n} \oplus \mc{O}(m)} \\
\ov{\mc{O}(m)} \cdot \ov{\T_{x,n}} &=  \ov{\T_{x,n} \oplus \mc{O}(m)}  \\
\ov{\T_{x,n}}\cdot \ov{\T_{x,m} } &= \ov{\T_{x,n+m}} + \ov{\T_{x,n} \oplus \T_{x,m}} \textrm{ for } x=0, \infty, \; m \neq n \\
\ov{\T_{x,n}} \cdot \ov{\T_{x',m}} &=  \ov{\T_{x,n} \oplus \T_{x',m}} \textrm{ for } x \neq x' \\
\ov{\mc{C}_n} \cdot \ov{\F} &= \ov{\mc{C}_n \oplus \mc{F}} \textrm{ if $\F$ is indecomposable and } \ov{\mc{F}} \neq \ov{\mc{C}_n}
\end{align*}
\end{theorem}

\bigskip

\begin{corollary}\label{cor:identities}
From Theorem \ref{HP1_identities} 
we deduce the following commutation relations:
\begin{align*}
\left[ \ov{\mc{O}(n)}, \ov{\mc{O}(m)} \right] &= 0 \\
\left[  \ov{\T_{x,n}}, \ov{\mc{O}(m)}  \right]  &= [\mc{O}(m+n)] \\ 
\left[  \ov{\T_{x,n}}, \ov{\T_{x',m}}    \right]  & = 0  \\
\left[  \ov{\mc{C}_n}, \ov{\F} \right] &= 0 \; \textrm{ for any indecomposable } \F.  
\end{align*}
\end{corollary}


\bigskip

Let $L \mathfrak{gl}_2 := \mathfrak{gl}_2 \otimes \mathbb{C}[t,t^{-1}]$, and $L \mathfrak{gl}^+_2 = (\mathfrak{d} \otimes t \mathbb{C}[t]) \oplus (e \otimes \mathbb{C}[t,t^{-1}])$ where $$\mathfrak{d} := \on{span} \left\{ h_1= \left[ \begin{matrix} 1 & 0 \\ 0 & 0 \end{matrix}\right], h_2=\left[ \begin{matrix} 0 & 0 \\ 0 & 1 \end{matrix}\right] \right\} \textrm{ and } e = \left[ \begin{matrix} 0 & 1 \\ 0 & 0 \end{matrix}\right].$$ Let $\kappa$ be the abelian Lie algebra with generators $\kappa_n, \; n \in \mathbb{N}$. Define 
\[
\rho : L \mathfrak{gl}^+_2 \oplus \kappa \rightarrow \H
\]
by setting $\rho(e \otimes t^k) = \ov{\mc{O}(k)}$, $\rho(h_1 \otimes t^n) = \ov{ \T_{0,n}}$, $\rho(h_2 \otimes t^m) = - \ov{\T_{0,m}}$, and $\rho(\kappa_n) = \ov{\mc{C}_n}$. 
The following theorem now follows from the commutation relations 
in Corollary \ref{cor:identities}. 
\goodbreak

\begin{theorem} \label{isom_theorem}
$\rho$ is an isomorphism of Lie algebras. Consequently, $\H \simeq \mathbb{U} (L \mathfrak{gl}^+_2 \oplus \kappa ).$
\end{theorem}

The sub-algebra $\wt{\H}$ of $\H$ generated by $\ov{\mc{O}(k)}, \ov{\T_{0,n}} + \ov{\T_{\infty,n}}$ is analogous to the sub-algebra of the Ringel-Hall algebra studied by Kapranov in \cite{Kap2}. It is easily seen to be isomorphic to $\mathbb{U}(L \mathfrak{sl}^+_2)$, with $L \mathfrak{sl}^+_2 = (\mathfrak{h} \otimes t \mathbb{C}[t] ) \oplus (e \otimes \mathbb{C}[t,t^{-1}])$, where
\[
\mathfrak{h} = \mathbb{C} \cdot  \left[ \begin{matrix} 1 & 0 \\ 0 & -1 \end{matrix}\right].
\]

\end{document}